\documentclass[11pt,reqno]{amsart}
\usepackage{bm,amsmath,amsthm,amssymb,mathtools,verbatim,amsfonts,tikz-cd,mathrsfs}
\usepackage{enumitem}

\usepackage[hidelinks,colorlinks=true,linkcolor=blue, citecolor=black,linktocpage=true]{hyperref}
\usepackage{graphicx}
\usepackage[alphabetic]{amsrefs}

\usepackage[top=1.4in, bottom=1.2in, left=1.4in, right=1.4in,marginpar=1in]{geometry}

\usetikzlibrary{matrix,arrows,decorations.pathmorphing}

\usepackage{chngcntr}
\counterwithin{figure}{section}

\newtheorem{thm}{Theorem}[section]
\newtheorem{prop}[thm]{Proposition}

\newtheorem{cor}[thm]{Corollary}
\newtheorem{lem}[thm]{Lemma}

\theoremstyle{definition}
\newtheorem{define}[thm]{Definition}

\theoremstyle{remark}
\newtheorem{rem}[thm]{Remark}

\newcommand{\ve}[1]{\boldsymbol{\mathbf{#1}}}

\newcommand{\R}{\mathbb{R}}

\newcommand{\Z}{\mathbb{Z}}

\newcommand{\Q}{\mathbb{Q}}

\renewcommand{\d}{\partial}
\renewcommand{\subset}{\subseteq}

\newcommand{\iso}{\cong}

\DeclareMathOperator{\GL}{{GL}}
\DeclareMathOperator{\gr}{{gr}}

\DeclareMathOperator{\id}{{id}}

\DeclareMathOperator{\Pin}{{Pin}}

\DeclareMathOperator{\Sp}{{Sp}}

\DeclareMathOperator{\Spin}{{Spin}}

\renewcommand{\P}{\bP}

\newcommand{\bF}{\mathbb{F}}

\newcommand{\bP}{\mathbb{P}}

\newcommand{\bT}{\mathbb{T}}

\newcommand{\RP}{\mathbb{RP}}

\newcommand{\frI}{\mathfrak{I}}

\newcommand{\frs}{\mathfrak{s}}

\newcommand{\CF}{\mathit{CF}}

\newcommand{\CFK}{\mathit{CFK}}

\newcommand{\xs}{\ve{x}}

\renewcommand{\a}{\alpha}
\renewcommand{\b}{\beta}

\newcommand{\veps}{\varepsilon}

\renewcommand{\GL}{\mathit{GL}}

\usepackage{leftidx}

\newcommand{\SL}{\mathit{SL}}

\newcommand{\Sss}[1]{\scriptscriptstyle{#1}}

\numberwithin{equation}{section}

\newcommand{\ar}{\mathrm{a.r.}}

\allowdisplaybreaks

\title{A note on the involutive invariants of splices}

\author{Kristen Hendricks}
\thanks{KH was partially supported by NSF grant DMS-2019396.}
\address{Department of Mathematics, Rutgers University, New Brunswick, NJ, USA}
\email{kristen.hendricks@rutgers.edu}
\author{Matthew Stoffregen}
\address{Department of Mathematics, Michigan State University, East Lansing, MI, USA}
\email{stoffre1@msu.edu}
\thanks{MS was partially supported by NSF grant DMS-2203828.}
\author{Ian Zemke}
\address{Department of Mathematics\\Princeton University\\  Princeton, NJ, USA}
\email{izemke@math.princeton.edu}
\thanks{IZ was partially supported by NSF grant DMS-2204375 and a Sloan Research Fellowship.}

\begin{document}

\begin{abstract} A natural family of potentially 2-torsion elements in the integer homology cobordism group consists of splices of knots with their mirrors. We show that such 3-manifolds have locally trivial involutive Floer homology. We show some related families of splices also have locally trivial involutive Floer homology. Our arguments show that many gauge theoretic invariants also vanish on these 3-manifolds.
\end{abstract}

\maketitle

\tableofcontents

\section{Introduction}

The integer homology cobordism group $\Theta^3_{\mathbb Z}$ is the group of oriented homology three-spheres up to the equivalence relation of homology cobordism. In 2013, C.~Manolescu used a $\mathrm{Pin}(2)$-equivariant version of Seiberg-Witten Floer homology to show that $\Theta^3_{\mathbb Z}$ contains no element $Y$ of order two having Rokhlin invariant $\mu(Y)=1$ \cite{ManolescuPin2Triangulation}, which due to previous work of Galewski--Stern \cite{GalewskiStern} and Matumoto \cite{Matumoto} was sufficient to disprove the remaining outstanding cases of the Triangulation Conjecture. It remains unknown whether $\Theta^3_{\mathbb Z}$ has any torsion elements; in particular, whether it contains a torsion element of order two. In order to produce an element of order two, it suffices to exhibit an oriented integer homology sphere $Y$ with an orientation-reversing diffeomorphism $Y\iso -Y$ 
 with the property that $Y$ takes a nontrivial value under any invariant of homology cobordism.

Three-manifolds obtained by splicing knot complements have attracted attention as a potential source of examples of order two elements in $\Theta_\Z^3$. If $K_0\subset Y_0$ and $K_1\subset Y_1$ are two knots, a \emph{splice} of $K_0$ and $K_1$ is a 3-manifold obtained by gluing $Y_0\setminus \nu(K_0)$ and $Y_1\setminus \nu(K_1)$ using an orientation reversing diffeomorphism $\phi$ of their boundaries. (Some authors require $\phi$ to swap the meridian with the Seifert longitude, but we consider more general diffeomorphisms $\phi$).

In this note, we consider two natural constructions of splices which produce homology 3-spheres with orientation reversing diffeomorphisms:
\begin{enumerate}[label=(Type-\arabic*), ref=\arabic*,leftmargin=1.8cm]
\item\label{sec:type-1-splice} Splices of a knot $K\subset Y$ with its mirror $K\subset -Y$ such that there is a diffeomorphism of the splice which swaps $Y\setminus \nu(K)$ and $-Y\setminus \nu(K)$.
\item\label{sec:type-2-splice} Splices of two knots $K_0$, $K_1$ in $S^3$ such  that there is a diffeomorphism of the splice which fixes the subsets $S^3\setminus \nu(K_0)$ and $S^3\setminus \nu(K_1)$ setwise, but is orientation reversing on each.
\end{enumerate}
Not all knots or gluing maps will yields a splice which admits such a symmetry. We enumerate all the requirements in Section~\ref{sec:splice-basics}.  In this paper, we will refer to the above families as \emph{Type-1} and \emph{Type-2 symmetric splices}, respectively.  We will see in Section~\ref{sec:type-1} that Type-1 splices must have $(Y,K)$ \emph{reversible}, i.e. there must be a diffeomorphism $(Y,K)\iso (Y,-K)$, where $-K$ denotes the knot with the opposite string orientation. In Section~\ref{sec:type-2} we will see that in a Type-2 symmetric splice, one of $K_0$ and $K_1$ must be negatively amphichiral and the other positively amphichiral. We will also enumerate all possible gluing maps.

\emph{Involutive Heegaard Floer homology} is a shadow theory of $\mathrm{Pin}(2)$-equivariant Seiberg-Witten Floer homology, introduced by Manolescu and the first author in 2015  and elaborated by the first and third authors with Manolescu \cite{HMInvolutive,HMZConnectedSum}, which is conjecturally equivalent to $\mathbb Z/4\mathbb Z$-equivariant Seiberg-Witten Floer homology. Although involutive Heegaard Floer homology does not possess the technical power of a $\mathrm{Pin}(2)$-equivariant theory, it enjoys the greater computability of the Heegaard Floer invariants, including a conveniently computable surgery formula \cite{HHSZExact}, and has been a key tool in recent developments regarding the structure of the homology cobordism group \cite{hendricks-hom-lidman, DHSThomcob, HHSZExact, HHSZQuotient}.

In this note we show that the homology cobordism involutive invariants of Type-1 and Type-2 splices are typically trivial. Recall for a homology sphere $Z$, these invariants consist of a pair $(\CF^-(Z), \iota)$ a chain complex together with a homotopy involution, together called an iota-complex; for more detail, see Section~\ref{sec:hfi}.

\begin{thm}\label{thm:symmetric-splices}
\item
\begin{enumerate}
\item Suppose that $K$ is a knot in an integer homology 3-sphere $Y$. If $Z$ is a Type-1 symmetric splice of $(Y,K)$ with $(-Y,-K)$ then the iota-complex $(\CF^-(Z), \iota)$ is locally trivial.
\item If $Z$ is a Type-2 symmetric splice of $(S^3,K_0)$ and $(S^3,K_1)$ such that $\CFK^-(K_0)$ and $\CFK^-(K_1)$ are (non-involutively) locally trivial, then the iota-complex $(\CF^-(Z), \iota)$ is locally trivial.
\end{enumerate}
\end{thm}

\begin{rem}
\item
\begin{enumerate}
\item It is not clear to the authors whether  Theorem~\ref{thm:symmetric-splices}(2) extends to all amphichiral $K_0$ and $K_1$, nor whether it can be extended to knots in homology 3-spheres other than $S^3$.
\item  As we mentioned above, in a symmetric splice of Type-2, one of $K_0$ and $K_1$ must be negative amphichiral, and the other positive amphichiral. To the best of our knowledge, all known amphichiral knots have locally trivial (non-involutive) knot Floer complex $\CFK^-(K)$. Also note if $K$ is \emph{strongly} negative amphichiral, i.e. if the pair $(S^3,K)$ admits an orientation reversing diffeomorphism $\phi$ which has exactly two fixed points, both of which lie along $K$, then Kawauchi's result \cite{Kawauchi_Snak} implies that $K$ is rationally slice, and hence has locally trivial $\CFK^-(K)$. 
\end{enumerate}
\end{rem}

The key topological input to Theorem~\ref{thm:symmetric-splices} is the following result:
\begin{prop}\label{prop:topology-intro}
\item
\begin{enumerate}
\item If $Z$ is a Type-1 symmetric splice of $(Y,K)$ with $(-Y,-K)$, then there is a negative definite $\Spin$ cobordism from $Z$ to $\RP^3$ which has $b_2^-=1$ and $b_1=0$.
\item If $Z$ is a Type-1 symmetric splice of $(Y,K)$ with $(-Y,-K)$, then there is a negative definite (non-$\Spin$) filling of $Z$ with $b_2^-=2$ and $b_1=0$.
\item If $Z$ is a Type-2 symmetric splice of $(Y_0,K_0)$ with $(Y_1,K_1)$, then there is a negative definite $\Spin$ cobordism from $Z$ to $(Y_0\# Y_1)_{-2}(K_0\# K_1)$ with $b_2^-=1$ and $b_1=0$. 
\end{enumerate}
\end{prop}

\begin{rem} In unpublished work, Mike Miller Eismeier independently proved Proposition~\ref{prop:topology-intro}(2) and used it to show that certain instanton theoretic gauge theoretic invariants are trivial on such splices.
\end{rem}

We now sketch some ideas in the proof of Theorem~\ref{thm:symmetric-splices}, assuming Proposition~\ref{prop:topology-intro}. For Type-1 splices, the negative definite $\Spin$ cobordism $W$ from $Z$ to $\RP^3$ has a unique self-conjugate $\Spin^c$ structure $\frs$. Furthermore, the Heegaard Floer grading shift $d(W,\frs)$ is equal to the correction term $d(\RP^3, \frs|_{\RP^3})$. Since $\RP^3$ is a Heegaard Floer L-space, the cobordism map $F_{W,\frs}$ can be viewed as a local map from $(\CF^-(Z),\iota)$ to the trivial $\iota$-complex. Since $Z\iso -Z$, we can dualize the map $F_{W,\frs}$ to get a local map in the opposite direction.

For Type-2 splices, the cobordism $W$ from $Z$ to $S^3_{-2}(K_0\# K_1)$ also has a unique self-conjugate $\Spin$ structure $\frs$. In this case, $\frs$ restricts to the $\Spin^c$ structure identified with $[1]\in \Z/2\iso \Spin^c(S^3_{-2}(K_0\# K_1))$ under the standard identification. We use \cite{HHSZExact}*{Theorem 1.6(2)}, which implies that since $\CFK^-(K_0)$ and $\CFK^-(K_1)$ are (non-involutively) locally trivial, the $\iota$-complex $(\CF^-(S^3_{-2}(K_0\# K_1),[1]),\iota)$ is locally trivial up to an overall grading shift.

\subsection{Other gauge theoretic invariants}

We note that the topological argument yielding Theorem~\ref{thm:symmetric-splices} also applies equally well to the $\Pin(2)$-equivariant Seiberg-Witten Floer spectra:

\begin{prop}\label{prop:pin2} The $\Pin(2)$-equivariant Seiberg-Witten Floer spectra of symmetric splices of Type-1 are locally trivial.
\end{prop}

The same argument also implies the vanishing of Lin's invariants  $\a(Z),\b(Z),\gamma(Z)$ \cite{LinPin2Monopole} for Type-1 symmetric splices.

Note that Proposition~\ref{prop:topology-intro}(2) can be applied to other invariants of gauge theory. In particular if $A$ is a partially ordered set and
\[
\omega\colon \Theta_{\Z}^3\to A
\]
is a homology cobordism invariant which is monotonic under negative definite cobordisms with $b_1=0$, then Proposition~\ref{prop:topology-intro}(2) implies that if $Z$ is a Type-1 symmetric splice, then $\omega(Z)\le \omega(S^3)$ and $\omega(S^3)\le \omega(Z)$, so in particular $\omega(Z)=\omega(S^3)$.
 This holds for many gauge theoretic invariants of homology cobordism.  For example, this argument applies to the $r_s$-invariants of Nozaki, Sato, Taniguchi  \cite{NST_Homology_cobordism} and Daemi's $\Gamma$ invariant \cite{Daemi_Gamma}. Compare also \cite{ScadutoDaemi}.

\subsection*{Organization} This note is organized as follows. In Section~\ref{sec:splice-basics} we discuss the geometry of splices; in particular, we classify which symmetric splices of Types 1 and 2 are integer homology spheres and have a natural orientation-reversing diffeomorphism. In Section~\ref{sec:hfi} we briefly recall relevant aspects of Heegaard Floer theory, focusing on its interaction with homology cobordism and concordance, for the reader's convenience. In Section~\ref{sec:proofs} we conclude with the proof of Theorem~\ref{thm:symmetric-splices}.

\subsection*{Acknowledgments} We are grateful to Jen Hom for helpful conversations, and to the anonymous referee for helpful comments and corrections.

\section{Symmetric splices} \label{sec:splice-basics}

In this section we recall some background on splices. Let $K_0\subset Y_0$ and $K_1\subset Y_1$ be oriented knots in integer homology 3-spheres. Let $\phi\in \Z^2\to \Z^2$ be a $2\times 2$ matrix with determinant $-1$. The map $\phi$ determines an orientation-reversing diffeomorphism
\[
\phi\colon \d (Y_0\setminus \nu(K_0))\to \d (Y_1 \setminus \nu(K_1))
\]
where we view the first component of $\Z^2$ being an oriented meridian of $K_i$, and the second component being the oriented Seifert longitude.
We define
\[
\Sp_\phi(K_0,K_1):=(Y_0\setminus \nu(K_0))\cup_\phi (Y_1\setminus \nu(K_1)).
\]

In this section, we will be interested in knots that have various symmetries. We use the following standard terminology:

\begin{define} Let $K$ be a knot in an oriented 3-manifold $Y$.
\begin{enumerate}
\item $(Y,K)$ is \emph{reversible} if $(Y,K)\iso (Y,-K)$.
\item  $(Y,K)$ is \emph{negative amphichiral} if $(Y,K)\iso (-Y,-K)$.
\item  $(Y,K)$ is \emph{positive amphichiral} if $(Y,K)\iso (-Y,K)$. 
\end{enumerate}
In the above, $\iso$ means orientation preserving diffeomorphic.
\end{define}

\subsection{Type-1 symmetric splices}
\label{sec:type-1}
We now focus on Type-1 splices, i.e. splices of $(Y,K)$ and $(-Y,-K)$ which admit orientation reversing diffeomorphisms which switch $Y\setminus \nu(K)$ with $-Y\setminus \nu(K)$ but fix $\bT^2:=\d (Y\setminus \nu(K))$ setwise. We will write $\Sp_\phi(K,mK)$ for such splices. In this section we prove the following:

\begin{prop}
\label{prop:type-1-topology} Suppose that $K$ is a knot in an integer homology 3-sphere $Y$, and $\phi\in \GL_2^-(\Z)$. Assume that $K$ is not an unknot in $Y$. Then the 3-manifold $\Sp_{\phi}(K,mK)$ admits an orientation-reversing diffeomorphism $g$ which fixes $\d (Y\setminus \nu(K))$ setwise and such that the image of each of $Y\setminus \nu(K)$ and $-Y\setminus \nu(K)$ is the other, if and only if the following are satisfied:
\begin{enumerate}
\item $K$ is reversible.
\item  
$
\phi=\begin{pmatrix} n &\pm 1\\
\pm (1+n^2)& n
\end{pmatrix}
$ for some $n\in \Z$. In this case we denote $\phi = \phi^{\pm}_n$. 
\end{enumerate}
\end{prop}
\begin{proof} It will be evident from the proof that if $(Y,K)$ and $\phi$ are as in the statement, then $\Sp_\phi(K,mK)$ will have a symmetry $g$ as above. Hence, we assume that $\Sp_{\phi}(K,mK)$ admits an orientation-reversing diffeomorphism $g$, as in the statement, and we will show that $K$ and $\phi$ have the stated properties.

 For our proof, it is somewhat easier write the gluing map in terms of a different basis. Note that if $(\mu,\lambda)$ is our oriented basis for $\d Y\setminus \nu(K)$, then $\phi$ is written in terms of the basis $(\mu,-\lambda)$ for $-Y\setminus \nu(K)$. For our purposes, it is more helpful to write $\phi$ in terms of the basis $(\mu,\lambda)$ for both $Y\setminus \nu(K)$ and $-Y\setminus \nu(K)$, using the same longitude and meridian for both.  Let us write $\psi$ for the map $\phi$ in this basis. Note that $\det(\psi)=1$.

 Additionally, to simplify the notation, we will view $\Sp_{\phi}(K,mK)$ as the union of two copies of $Y\setminus \nu(K)$, which we denote by $X_0$ and $X_1$. We write
\[
\Sp_{\phi}(K,mK)=\frac{ X_0\sqcup X_1}{\sim}
\]
where $x\in \d X_0$ is identified with $\psi(x)\in \d X_1$. By assumption $g$ is
induced by some diffeomorphisms $g_{10}\colon X_0\to X_1$ and $g_{01}\colon X_1\to X_0$ as below:
\[
\begin{tikzcd}
X_0 \ar[r, "g_{10}", bend left] & X_1\ar[l, "g_{01}", bend left].
\end{tikzcd}
\]
The diffeomorphism $g_{10}\sqcup g_{01}$ descends to the quotient if and only if
\begin{equation}
g_{10}(x)=(\psi\circ  g_{01}\circ  \psi)(x)
\label{eq:compatibility-phi-d}
\end{equation}
for all $x\in \d X_0$.

We now claim that 
\begin{equation}
g_{01}|_{\d X_1},g_{10}|_{\d X_0}\in \{\id,-\id\}. \label{eq:d01-pm-id}
\end{equation}
 To see this, note that both must map $\lambda$ to $\pm \lambda$, because they must preserve the kernel of the map $H_1(\d X_i)\to H_1(X_i)$. Less obviously, they must also send $\mu$ to $\pm \mu$. Homology considerations imply that $g_{01}$ and $g_{10}$ map $\mu$ to $\pm \mu+j \lambda$ for some $j\in \Z$. This would imply that, up to composition with the elliptic involution, $g_{01}|_{\d X_0}$ is an $j$-fold composition of a Dehn twist parallel to the Seifert longitude, and similarly for $g_{10}$. It follows from  \cite{McCullough_Boundary_DehnTwist}*{Theorem~1}  that this can only happen if $\lambda$ bounds a disk in $Y\setminus \nu(K)$, i.e. $K$ is an unknot, which we exclude by hypothesis.

Since $g_{01}|_{\d X_{1}},g_{10}|_{\d X_{0}}\in \{\id,-\id\}$, these maps are central in $\GL_2(\Z)$, and hence Equation~\eqref{eq:compatibility-phi-d} implies that $\psi^2=\pm \id$. 

We now consider the map $\psi$ in more detail. The Mayer-Vietoris exact sequence reads
\[
H_1(\bT^2)\to H_1(Y\setminus \nu(K))\oplus H_1(Y\setminus \nu(K))\to H_1(Z)\to 0.
\]
In particular, we see that for $Y$ to be a homology sphere, we need $\psi(\lambda)=\pm \mu+n \lambda$, for some $n\in \Z$. That is, we can write $\psi$ as a matrix as
\[
\psi=\begin{pmatrix} n_1 &\pm 1\\
*& n_2
\end{pmatrix}
\]
The condition that $\det \psi=1$ imposes the restriction that
\begin{equation}
\psi=\begin{pmatrix} n_1 &\pm 1\\
\mp (1-n_1n_2)& n_2
\end{pmatrix}.
\label{eq:matrix-phi}
\end{equation}
It is straightforward to see that there are no such matrices of the above form which square to the identity matrix. On the other hand, the matrix in Equation~\eqref{eq:matrix-phi} squares to $-\id$ if and only if $n_1=-n_2$. Setting $n=n_1$ and then changing to the basis gives the expression for $\phi$ in the statement.

Next, we observe that Equation~\eqref{eq:compatibility-phi-d} now implies that $g_{10}|_{\d X_0}=-g_{01}|_{\d X_1}$. Therefore one map is the identity, while the other is the elliptic involution. Therefore $Y\setminus \nu(K)$ admits an orientation preserving diffeomorphism which restricts to the elliptic involution on the boundary. Equivalently, there is a diffeomorphism of pairs $(Y,K)\iso (Y,-K)$. The proof is therefore complete.
\end{proof}

\begin{lem}
\label{lem:change-sign} Suppose that $(Y,K)$ is reversible. Then $\Sp_{\phi_n^{+}}(K,mK)\iso \Sp_{\phi_{-n}^-}(K,mK)$.
\end{lem}
\begin{proof} Since $K$ is reversible, the elliptic involution of the boundary extends to an orientation preserving diffeomorphism of $Y\setminus \nu(K)$. Therefore
\[
\Sp_{\phi_n^+}(K,mK)\iso \Sp_{-\phi_n^+}(K,mK)=\Sp_{\phi_{-n}^-}(K,mK)
\]
completing the proof.
\end{proof}

\subsection{Type-2 symmetric splices}
\label{sec:type-2}

We now consider Type-2 splices.  We say that a pair $(Y,K)$ is \emph{negative amphichiral} if $(-Y,-K)\iso (Y,K)$, and we say that $(Y,K)$ is \emph{positive amphichiral} if $(-Y,K)\iso (Y,K)$. 

\begin{prop}
\label{prop:type-2-topology}
Let $(Y_0,K_0)$ and $(Y_1,K_1)$ be knots and $\phi\in \GL^-_2(\Z)$, and furthermore suppose that $\Sp_\phi(K_0,K_1)$ is an integer homology 3-sphere. Then $\Sp_{\phi}(K_0,K_1)$ admits an orientation-reversing diffeomorphism $g$ which preserves the subspaces $Y_i\setminus \nu(K_i)$ setwise if and only if the following hold:
\begin{enumerate}
\item One of $(Y_0,K_0),$ $(Y_1,K_1)$ is negative amphichiral and the other is positive amphichiral.
\item $
\phi=\pm \begin{pmatrix}0 &1\\
1&0
\end{pmatrix}.$
\end{enumerate}
\end{prop}
\begin{proof} We write
\[
X_0=Y_0\setminus \nu(K_0)\quad \text{and} \quad X_1=Y_1\setminus \nu(K_1).
\]
It will follow from the course of our proof that if $\phi$ and $(Y_i,K_i)$ are as in the statement, then there is an orientation-reversing diffeomorphism $g$ as in the statement. Hence we will assume that such a $g$ exists, and prove that it has the stated form.

We assume that $g$ is induced by a pair of maps, $g_0$ and $g_1$, as below
\[
\begin{tikzcd}
X_0\ar[loop left, "g_0"] &[-.7cm] X_1 \ar[loop right, "g_1"]
\end{tikzcd}
\]
The maps $g_0$ and $g_1$ induce a map on the quotient space if and only if
\begin{equation}
\phi\circ g_0=g_1\circ \phi.\label{eq:phi-commutator}
\end{equation}
The proof of Equation~\eqref{eq:d01-pm-id} shows that for $i=0,1$, we have $g_i|_{\d X_i}\in \{e,-e\}$ where
\[
e=\begin{pmatrix} 1& 0\\
0& -1
\end{pmatrix}.
\]
Equation~\eqref{eq:phi-commutator} implies that
\[
\phi \circ e=\pm e \circ \phi.
\]
It is easy to see that this restricts $\phi$ to be one of four matrices:
\[
\phi=\pm e\qquad \text{and} \qquad \phi=\pm \begin{pmatrix}0 &1\\
1& 0
\end{pmatrix}.
\]
Note that if $\phi=\pm e$, then the splice $\Sp_{\phi}(K_0,K_1)$ has $b_1=1$, so we exclude this case and restrict to the second case. We observe that in the latter case, we have
\[
\phi \circ e=-e \circ\phi.
\]
In particular, we conclude from Equation~\eqref{eq:phi-commutator} that $g_0=-g_1$.
Note that this corresponds exactly to one of $K_0$ and $K_1$ being negative amphichiral, and the other being positive amphichiral.
\end{proof}

\subsection{Factorizations in $\SL_2(\Z)$}

In this section, we describe some straightforward algebra which will be used in the subsequent section on Kirby calculus.

We consider the elements $\psi_n^+\in \SL_2(\Z)$, given by 
\[
\psi_n^+=\begin{pmatrix} n & 1\\
- (1+n^2)& -n
\end{pmatrix}
\]
We define the following elements of $\SL_2(\Z)$:
\[
T_n:=
\begin{pmatrix}
1 &n\\
0& 1
\end{pmatrix}\quad \text{and} \quad H=\begin{pmatrix} 0&  1\\
- 1& 0
\end{pmatrix}.
\]

\begin{lem}
\label{lem:factorization-phi-n} The map $\psi_n^{+}$ may be written as
\[
\psi_n^{+}=H T_{-n}  H  T_n H.
\]
\end{lem}
The proof is a straightforward computation, which we leave to the reader.


\subsection{Kirby calculus}

We can now translate Lemma~\ref{lem:factorization-phi-n} into Kirby calculus. Our main result is the following:

\begin{prop}
Let $K\subset Y$ be a knot. The manifold $\Sp_{\phi_n^+}(K,mK)$ has a Kirby diagram as shown in Figure~\ref{fig:Kirby-Y-n}. 
\end{prop}

We begin with the following elementary topological lemma.
\begin{lem} Let $(Y_0,K_0)$ and $(Y_1,K_1)$ be knots with Morse framings $\lambda_0$ and $\lambda_1$ respectively. Let $\phi$ be the gluing map which identifies the meridian $\mu_0$ with $\mu_1$, and which maps $\lambda_0$ to $-\lambda_1$. (Here, $-\lambda_2$ denotes the Morse framing $\lambda_2$ with the parametrization reversed). Then 
\[
(Y_0\setminus \nu(K_0) )\cup_\phi (Y_1\setminus \nu(K_1))
\]
is equal to $(Y_0\# Y_1)_{\lambda_0+\lambda_1}(K_0\# K_1)$. 
\end{lem}

See \cite{GordonSatellite}*{Lemma~7.1} or \cite{Fukuhara-Maruyama_additive}*{Lemma~6.1} for a proof. See also \cite{Saveliev_Homology_3-spheres}*{Section~1.1.7}.

The above lemma extends in a straightforward manner to link complements when we take the connected sum along a single component. We note that the Hopf link has complement $\bT^2\times [0,1]$. The meridian of the first component of the Hopf link is equal to the longitude of the second up to sign, and vice versa. Therefore from a factorization of the gluing diffeomorphism $\psi_n^+$, as in Lemma~\ref{lem:factorization-phi-n}, we may read off a Kirby calculus description of $\Sp_{\phi_n^+}(K,mK)$. Namely, we start with $K$, which we give framing $0$. Reading the factorization 
\[
\psi_n^+=H T_{-n} H T_n H
\]
from right to left we form a Kirby calculus presentation inductively as follows:
\begin{enumerate}
\item Start with $K$, given framing 0.
\item For each $H$, we take the connected sum with a $(0,0)$-framed Hopf link. \item For each $T_n$, we add $n$ to the framing of the most-recently added component in this process.
\item We finish by taking the connected sum of the final unknot, which has framing 0, with $mK$.
\end{enumerate} 

Note that since we are assuming that $K$ is reversible, we do not need to worry about the sign of the clasps that we add when taking the connected sum with Hopf links.

\begin{figure}[h]
\begingroup%
  \makeatletter%
  \providecommand\color[2][]{%
    \errmessage{(Inkscape) Color is used for the text in Inkscape, but the package 'color.sty' is not loaded}%
    \renewcommand\color[2][]{}%
  }%
  \providecommand\transparent[1]{%
    \errmessage{(Inkscape) Transparency is used (non-zero) for the text in Inkscape, but the package 'transparent.sty' is not loaded}%
    \renewcommand\transparent[1]{}%
  }%
  \providecommand\rotatebox[2]{#2}%
  \newcommand*\fsize{\dimexpr\f@size pt\relax}%
  \newcommand*\lineheight[1]{\fontsize{\fsize}{#1\fsize}\selectfont}%
  \ifx\svgwidth\undefined%
    \setlength{\unitlength}{135.10420153bp}%
    \ifx\svgscale\undefined%
      \relax%
    \else%
      \setlength{\unitlength}{\unitlength * \real{\svgscale}}%
    \fi%
  \else%
    \setlength{\unitlength}{\svgwidth}%
  \fi%
  \global\let\svgwidth\undefined%
  \global\let\svgscale\undefined%
  \makeatother%
  \begin{picture}(1,0.41806873)%
    \lineheight{1}%
    \setlength\tabcolsep{0pt}%
    \put(0,0){\includegraphics[width=\unitlength,page=1]{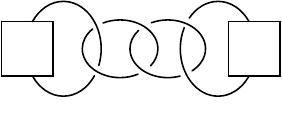}}%
    \put(0.20915357,0.01023197){\makebox(0,0)[lt]{\lineheight{1.25}\smash{\begin{tabular}[t]{l}$0$\end{tabular}}}}%
    \put(0.74483281,0.01023197){\makebox(0,0)[lt]{\lineheight{1.25}\smash{\begin{tabular}[t]{l}$0$\end{tabular}}}}%
    \put(0.41600692,0.09101465){\makebox(0,0)[t]{\lineheight{1.25}\smash{\begin{tabular}[t]{c}$n$\end{tabular}}}}%
    \put(0.58428877,0.09101465){\makebox(0,0)[t]{\lineheight{1.25}\smash{\begin{tabular}[t]{c}$-n$\end{tabular}}}}%
    \put(0.09426729,0.21855462){\makebox(0,0)[t]{\lineheight{1.25}\smash{\begin{tabular}[t]{c}$K$\end{tabular}}}}%
    \put(0.90074647,0.21855462){\makebox(0,0)[t]{\lineheight{1.25}\smash{\begin{tabular}[t]{c}$mK$\end{tabular}}}}%
  \end{picture}%
\endgroup%

\caption{The manifold $\Sp_{\phi_n^+}(K,mK)$. We view the box labeled by $K$ as being inside $Y$.}
\label{fig:Kirby-Y-n}
\end{figure}
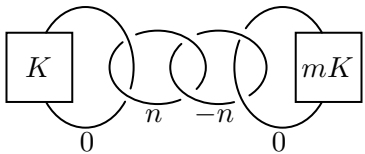

%

We now describe some Kirby calculus moves in our description of $\Sp_{\phi_n^+}(K,mK)$ which will be helpful later on.

\begin{lem}
\label{lem:reduce-n} Let $K$ be a reversible knot in an integer homology sphere $Y$, and $n\in \Z$. Then there are reversible knots $K'\subset Y'$ and $K''\subset Y''$, where $Y'$ and $Y''$ are also integer homology 3-spheres, so that
\[
\Sp_{\phi_n^+}(K,mK)=\Sp_{\phi_{n+1}^+}(K',mK')=\Sp_{\phi_{n-1}^+}(K'',mK'').
\]
\end{lem}
\begin{proof} The proof is to take the Kirby calculus description in Figure~\ref{fig:Kirby-Y-n}, and blow up the clasps between $K$ and the unknotted component clasped with it and between $mK$ and the unknotted component clasped with it. In this manner, $n$ can be increased or decreased. Then $Y'=Y_{+1}(K)$ and $K'$ is the dual knot of $K$. We give $K'$ blackboard framing $1$ in Figure~\ref{fig:Kirby-Y-n-blown-up}, but note that this corresponds to the Seifert framing for $K'$ if we view it as living in $Y_{+1}(K)$. Hence, we obtain a description of the same form as Figure~\ref{fig:Kirby-Y-n}, except with $n$ replaced by $n+1$.
\end{proof}

\begin{figure}[h]
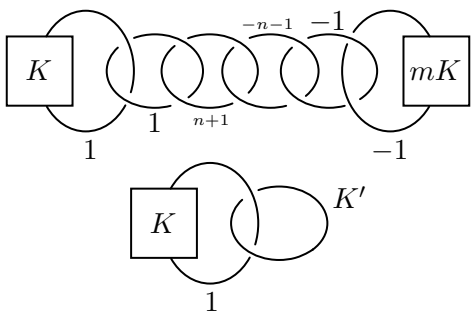
\caption{Top: An alternate surgery description of $\Sp_{\phi_n^+}(K,mK)$, obtained by blowing up two clasps in Figure~\ref{fig:Kirby-Y-n}. This identifies $\Sp_{\phi_n^+}(K,mK)$ with $\Sp_{\phi_{n+1}^+}(K',mK')$ where $K'\subset Y'=Y_{+1}(K)$ is the dual knot of $K$. Bottom: the knot $K'\subset Y_{+1}(K)$. }
\label{fig:Kirby-Y-n-blown-up}
\end{figure}

\section{Heeegaard Floer invariants of concordance and homology cobordism } \label{sec:hfi}

In this section, we review some background on Heegaard Floer invariants of homology cobordism and knot concordance. We focus on Hendricks and Manolescu's \emph{involutive Heegaard Floer homology} \cite{HMInvolutive}, which we review in Section \ref{sec:iota}, as well as the notion of a knot-like complex, which we review in Section \ref{sec:knot-like}. We presume the reader is familiar with ordinary Heegaard Floer homology for three-manifolds \cite{OSDisks, OSProperties} and knots \cite{OSKnots, RasmussenKnots}.

\subsection{Iota-complexes and involutive Heegaard Floer homology} \label{sec:iota}

In this section we briefly introduce the structure of the involutive Heegaard Floer invariants, with a focus on the properties of local equivalence. 

We begin with certain algebraic definitions. Throughout, $\mathbb F$ denotes the field with two elements, $U$ is a variable of degree $-2$, and $\bF[U]_d$ is the graded module such that $\gr(1) = d$.

 \begin{define}
 	\label{def:iota-complex}
 	An \emph{iota-complex} (or \emph{$\iota$-complex}) $(C,\iota)$ is a chain complex $C$, which is free and finitely generated over $\bF[U]$, equipped with an endomorphism $\iota$.  Here $\bF$ is the field of $2$ elements, and $U$ is a formal variable with grading $-2$.  Furthermore, the following hold:
 	\begin{enumerate}
 		\item\label{iota-1} $C$ is equipped with a $\Z$-grading, compatible with the action of $U$.  We call this grading the \emph{Maslov} or \emph{homological} grading.
 		\item\label{iota-2} There is a grading-preserving isomorphism $U^{-1}H_*(C)\iso \bF[U,U^{-1}]$.
 		\item \label{iota-3} $\iota$ is a grading-preserving chain map and $\iota^2\simeq \id$.
 	\end{enumerate}
 \end{define}
 
Given two iota-complexes $(C_1, \iota_1)$ and $(C_2, \iota_2)$, a homogeneously graded $\bF[U]$-chain map $f \colon C_1 \to C_2$ is said to be an \emph{$\iota$-homomorphism} if $\iota_2\circ  f+f\circ  \iota_1\simeq 0$. Two iota-complexes $(C_1,\iota_1)$ and $(C_2,\iota_2)$ are called \emph{$\iota$-equivalent} if there is a homotopy equivalence $\Phi\colon C_1\to C_2$ which is an $\iota$-homomorphism. 

Heegaard Floer homology associates to any closed oriented $3$-manifold $Y$ equipped with a $\Spin^c$ structure $\frs$ an $\bF[U]$-chain complex $\CF^-(Y,\frs)$, well defined up to homotopy equivalence. If $\frs$ is self-conjugate, involutive Heegaard Floer homology considers the additional data of a homotopy involution $\iota$ on $\CF^-(Y,\frs)$. In the case that $Y$ is a rational homology $3$-sphere, $(\CF^-(Y,\frs),\iota)$ is an iota-complex. Hendricks and Manolescu \cite{HMInvolutive} prove that pair $(\CF^-(Y,\frs),\iota)$ is well defined up to the notion of iota-equivalence described above.

Continuing with algebra, the tensor product of iota-complexes $(C_1,\iota_1)$ and $(C_2,\iota_2)$ is given by 
 \begin{equation}
 (C_1,\iota_1)\otimes(C_2,\iota_2):=(C_1\otimes_{\bF[U]}C_2,\iota_1\otimes \iota_2).
\label{eq:group-relation-I}
 \end{equation}
  Moreover, Hendricks, Manolescu, and Zemke \cite{HMZConnectedSum} establish that 
 \[
 (\CF^-(Y_1\# Y_2,\frs_1\# \frs_2),\iota)\simeq (\CF^-(Y_1,\frs_1),\iota_1)\otimes (\CF^-(Y_2,\frs_2),\iota_2),
 \]
 where $\simeq$ denotes homotopy-equivalence of iota-complexes.

 \begin{define}\label{def:local-equivalence} Suppose $(C,\iota)$ and $(C',\iota')$ are two iota-complexes.
 	\begin{enumerate}
 		\item A \emph{local map} from $(C,\iota)$ to $(C',\iota')$ is a grading-preserving $\iota$-homomorphism $F\colon C\to C'$, which induces an isomorphism from $U^{-1} H_*(C)$ to $U^{-1} H_*(C')$.
 		\item We say that $(C,\iota)$ are $(C',\iota')$ are \emph{locally equivalent} if there is a local map from $(C,\iota)$ to $(C',\iota')$, as well as a local map from $(C',\iota')$ to $(C,\iota)$. We say that $(C, \iota)$ is \emph{locally trivial} if it is locally equivalent to $(\bF[U]_0, \mathrm{Id})$.
 	\end{enumerate}
 \end{define}
 The set of local equivalence classes forms an abelian group, denoted $\frI$, with product given by the operation $\otimes$ in equation~\eqref{eq:group-relation-I} \cite{HMZConnectedSum}*{Section~8}. Inverses are given by dualizing both the chain complex $C$ and the map $\iota$ with respect to $\bF[U]$. The map
 \[
 Y\mapsto [(\CF^-(Y), \iota)]
 \]
 determines a homomorphism from $\Theta_{\Z}^3$ to $\frI$ \cite{HMZConnectedSum}*{Theorem~1.8}. 

The local equivalence classes of nonzero integer surgeries on knots are computed in \cite[Theorem 1.6(2)]{HHSZExact}. For our purposes, the important case is the following.

\begin{lem}\label{lem:local-class-even} \cite[Theorem 1.6(2)]{HHSZExact} For $n>0$, the local equivalence class of $(\CF^-(S^3_{2n}(K),[n]),\iota)$ has the form
  \[
 \begin{tikzcd}[column sep={1cm,between origins},labels=description] 
A_{n}^-(K)
	\ar[dr, "v"]
& & A_n^-(K) \ar[dl,"v"]\\
& B_n^-(K)
\end{tikzcd}
\]
 where $A_n^-(K)$ and $B_n^-(K)$ are subcomplexes of the knot Floer complex of $K$, $v$ is a particular map between them, and the involution swaps the two copies of $A_n^-(K)$, and fixes $B_n^-(K)$. The gradings on the above are induced by the Maslov grading on the knot Floer complex, shifted up by the Heegaard Floer correction term $d(L(2n,1),[n])$ of the lens space $L(2n,1)$ in the corresponding $\Spin^c$ structure.
\end{lem}

One straightforward corollary is the following:

\begin{cor}
\label{cor:locally-trivial-d=0} For $n>0$, if $K\subset S^3$ is a knot such that $d(A_n^-(K))=0$, then $(\CF^-(S_{2n}^3(K),[n]), \iota)$ is locally equivalent to $(\bF[U]_d,\id)$, where $d=d(L(2n,1),[n])$. 
\end{cor}
\begin{proof}  We note that $B_n^-(K)\simeq \bF[U]$, so using the same logic as in the proof of \cite[Proposition~3.24]{HHSZExact}, we may replace it with a copy of $\bF[U]$. By the classification theorem for finitely generated chain complexes over $\bF[U]$, we can write $A_n^-(K)$ as a sum of one tower $\bF[U]$, as well as some number of 2-step complexes of the form $\bF[U]\xrightarrow{U^i} \bF[U]$. Write $\xs_l$ and $\xs_r$ for tower generators of the two copies of $A_n^-(K)$, which is to say, generators for the copy of $\bF[U]$ in the basis chosen. The map $v$ sends $\xs_l$ and $\xs_r$ to a nonzero element of the tower $\bF[U]$. Since $d(A_n^-(K))=0$, we conclude that $v(\xs_l)=v(\xs_r)=1.$ On a two step subcomplex of the left copy of $A_n^-(K)$, say with generators $\ve{a}$ and $\ve{b}$ such that $\d(\ve{a})=U^i \ve{b}$, we must have $v(\ve{b})=0$ because $v$ is a chain map. If $v(\ve{a})$ is non-zero, then perform a change of basis, adding a multiple of $\ve{x}_l$ to $\ve{a}$. After this change of basis, $v(\ve{a})=0$. We do the same change of basis to the right copy of $A_n^-(K)$. After this change of basis, it becomes apparent that the complex in the statement of Lemma~\ref{lem:local-class-even} is locally equivalent to
\[
 \begin{tikzcd}[column sep={1cm,between origins},labels=description] 
\bF[U]_{d}
	\ar[dr, "1"]
& & \bF[U]_d\ar[dl,"1"]\\
& \bF[U]_d
\end{tikzcd},
\]
where the involution is reflection, and $d$ denotes $d(L(2n,1),[n])$. The above is homotopy equivalent to $(\bF[U]_d, \id)$, completing the proof.
\end{proof}

\subsection{Knot-like complexes} \label{sec:knot-like}

We now recall the standard notion of knot-like complexes in Heegaard Floer theory. There are many different variations on the definition in the literature due to many different authors. The earliest version is Hom's notion of $\veps$-equivalence \cite{HomEpsilon}. See also \cite{ZemConnectedSums, DHSTmore, DHST_Homology_Concordance} for other variations.

\begin{define} A \emph{knot-like} complex $C$ is a finitely generated, free chain complex over a 2-variable polynomial ring $\bF[U,V]$ satisfying the following:
\begin{enumerate}
\item $C$ is equipped with a $\Z\times \Z$-valued bigrading, denoted $(\gr_w,\gr_z)$, which has the property that $(\gr_w-\gr_z)/2$ is integrally valued. The variable $U$ has bigrading $(-2,0)$ and the variable $V$ has bigrading $(0,-2)$.
\item There is a grading preserving isomorphism $(U,V)^{-1} H_*(C)\iso \bF[U,V,U^{-1},V^{-1}]$.
\item $\d$ has bigrading $(-1,-1)$. 
\end{enumerate}
\end{define}

A local map from $C_0$ to $C_1$ (where $C_i$ are knot-like complexes) consists of a grading preserving $\bF[U,V]$ linear chain map $F\colon C_0\to C_1$ such that $F$ induces an isomorphism from $(U,V)^{-1} H_*(C_0)$ to $(U,V)^{-1} H_*(C_1)$. We say that $C_0$ and $C_1$ are \emph{locally equivallent} if there exist local maps from $C_0$ to $C_1$ and from $C_1$ to $C_0$.  A knot-like complex $C$ is \emph{locally-trivial} if it is locally equivalent to a rank one complex $\bF[U,V]$ wherein $1\in \bF[U,V]$ is concentrated in grading $(0,0)$. Note that $C$ is locally trivial if and only if there is an isomorphism
\[
C\iso \bF[U,V]\oplus A
\]
where $A$ is a summand of $C$ such that $(U,V)^{-1} H_*(A)\iso 0$. 

If $K\subset S^3$, the full version of the knot Floer complex $\CFK^-(K)$ is a knot-like complex. If $K$ is a slice knot, then $\CFK^-(K)$ is locally trivial.

\section{Proofs of the main results} \label{sec:proofs}

We first observe that if $Z$ is a symmetric splice, since $Z$ has order at most two in $\Theta^3_{\Z}$, we must have $d(Z)=0$ since $d$ is a homomorphism.

\begin{lem}
\label{lem:neg-def-spin-cobordism}
Suppose that $Z$ is the homology sphere obtained by splicing $(Y,K)$ and $(-Y,-K)$ using the gluing map $\phi_0^{+}$ from Proposition~\ref{prop:type-1-topology}, where $Y$ is a homology 3-sphere. Then there is a negative definite $\Spin$ cobordism $W$ from $Z$ to $\RP^3$. The 4-manifold $W$ has a unique self-conjugate $\Spin^c$ structure $\frs$. Furthermore, letting $d(W,\frs)$ denote the grading shift of the map associated to the cobordism $(W,\frs)$, we have
\[
d(W,\frs)=d(\RP^3, \frs|_{\RP^3})-d(Z).
\]
\end{lem}

\begin{proof}We begin with the Kirby calculus presentation from Figure~\ref{fig:Kirby-Y-n}. We can blow-down one of the unknots to obtain a Kirby calculus description of $Z$ as Dehn surgery on a knot $K\# H\#{-K}\subset Y\#{-Y}$, where $H$ denotes a Hopf link. That is, we add a clasp  between $K$ and $-K$. (Note that the sign of the clasp is not important since $(Y,K)$ is reversible). The two components are given framing $0$.  We now blow up the clasp to obtain the 3 component link $K\cup U\cup mK$ with clasps between the components. Each component is given framing $-1$. There is a cobordism $X$ from $Z$ to a manifold $Z'$ by performing $-1$ surgery on a meridian of the unknot $U$, where this knot is given Seifert framing $-2$ inside of $Z$. The result is $-2$ surgery on $K\# {-K}\subset Y\# {-Y}$. The pair $(Y\#{-Y}, K\# {-K})$ is homology concordant to $(S^3,U)$. Therefore $Z'$ admits a homology cobordism to $\R \P^3$, viewed as $-2$ surgery on the unknot. Let $W$ denote the composition of these two cobordisms. The cobordism $W$ is shown in Figure~\ref{fig:5}.

We observe that $d(Z)=0$ since $Z\iso-Z$. On the other hand, we compute that the shift in grading for the $\Spin$ structure on $W$ is 
\[
\frac{-2\chi(W)-3\sigma(W)}{4}=\frac{1}{4}.
\]
We note that the $d$ invariants of the two $\Spin^c$ structures ($=\Spin$ structures) on $\RP^3$ are $1/4$ and $-1/4$. Since the Maslov grading takes values in a single coset of $\Q/ \Z$ in each $\Spin^c$ structure, it follows that  the $\Spin$ structure on $W$ restricts to the $\Spin^c$ structure on $\RP^3$ which has $d$-invariant $1/4$.
\end{proof}

\begin{figure}[h]
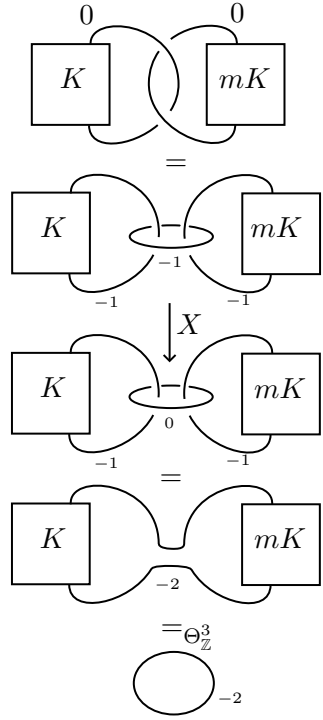
\caption{The cobordism $W$ from $Z$ to $\RP^3$.}
\label{fig:5}
\end{figure}

Using Lemma~\ref{lem:neg-def-spin-cobordism}, we now prove Theorem~\ref{thm:symmetric-splices}(1), which concerns symmetric splices of Type-1:

\begin{proof} [Proof of Theorem~\ref{thm:symmetric-splices}(1)] Let $Z$ be a symmetric splice of Type-1. By Proposition~\ref{prop:type-1-topology}, $Z$ can be written as a splice $\Sp_{\phi_n^{\pm}}(K,mK)$ for some pair $(Y,K)$, where $Y$ is a homology sphere and $(Y,K)$ is reversible. By Lemmas~\ref{lem:change-sign} and~\ref{lem:reduce-n}, we may assume that $n=0$ by changing $(Y,K)$ appropriately. Applying Lemma~\ref{lem:neg-def-spin-cobordism} gives a local map from $(\CF^-(Z), \iota)$ to $(\CF^-(\RP^3,\frs), \iota)$ for the $\Spin$ structure $\frs$ with $d(\RP^3, \frs)=1/4$. Since the grading shift of the cobordism map is also $1/4$, and $(\CF^-(\RP^3,\frs), \iota) \iso (\bF[U]_{1/4},\id)$, we conclude that there is a local map from $(\CF^-(Z), \iota)$ to the trivial complex. Dualizing and using the fact that $Z\iso -Z$ gives a local map in the opposite direction. \end{proof}

\begin{proof} [Proof of Theorem~\ref{thm:symmetric-splices}(2)] 
The proof is similar to the proof of Theorem~\ref{thm:symmetric-splices}(1). By Proposition~\ref{prop:type-2-topology},  the 3-manifold $Z$ can be written as $\Sp_{\phi_0^+}(K_0,K_1)$ where $K_0,K_1\subset S^3$ are positive and negative amphichiral knots, respectively. By adapting the argument from Theorem~\ref{thm:symmetric-splices}(1), we obtain a negative definite $\Spin$ cobordism from $Z$ to $S_{-2}^{3}(K_0\# K_1)$ which shifts the Maslov grading by $1/4$. Assuming (up to relabeling) that $K_0$ is positive amphichiral and $K_1$ is negative amphichiral, we observe that
\[
S^3_{-2}(K_0\# K_1)\iso -S^3_{+2}(K_0\# {-}K_1),
\]
where $-K_1$ denotes $K_1$ with its string orientation reversed. Therefore, Corollary~\ref{cor:locally-trivial-d=0} implies that
\[
(\CF^-(S^3_{+2}(K_0\# {-}K_1),[1]), \iota)\sim_{\mathrm{loc}} (\bF[U]_{-1/4}, \id).
\]
Dualizing, we obtain that
\[
(\CF^-(S_{-2}^3(K_0\# K_1),[1]), \iota)\sim_{\mathrm{loc}} (\bF[U]_{1/4}, \id).
\]
It follows that there is a local map from $(\CF^-(Z), \iota)$ to $(\bF[U],\id)$. Since $Z\iso -Z$, we conclude that $(\CF^-(Z), \iota)$ is locally trivial.
 \end{proof}

We now prove of Proposition~\ref{prop:topology-intro}, most of which we have already proven:
\begin{proof}[Proof of Proposition~\ref{prop:topology-intro}]
 Part (1) follows from Lemma~\ref{lem:neg-def-spin-cobordism}, above. Part~(3) is similar, and is described in our proof of Theorem~\ref{thm:symmetric-splices}(2). Finally Part~(2) is obtained by composing the cobordism from Part~(1) with the natural negative definite cobordism from $\RP^3$ to $\emptyset$, namely the disk bundle over $S^2$ with Euler number $-2$. 
\end{proof} 

\begin{proof}[Proof of Proposition~\ref{prop:pin2}]
	The argument is essentially identical to the proof of Part (1) of Theorem~\ref{thm:symmetric-splices}, but where the notation is adjusted to be for Seiberg-Witten Floer spectra in the setting of \cite{ManolescuPin2Triangulation}.  In particular, local equivalence of $\Pin(2)$-equivariant spectra is defined just as in Definition~\ref{def:local-equivalence} above, except $\Pin(2)$-equivariant spectra take the place of iota-complexes (see  \cite{StoffregenSeifertFibered}*{Definition 2.7} and surrounding discussion).  
	
	Let $Z$ be a symmetric splice of Type-1.  Lemma~\ref{lem:neg-def-spin-cobordism} gives a local map \[\Sigma^{\frac{1}{16}\mathbb{H}}\mathit{SWF}(Z)\to \mathit{SWF}(\mathbb{RP}^3,\frs),\]
	with $\frs$ as in the proof of Theorem~\ref{thm:symmetric-splices}. We refer the reader to \cite{ManolescuSWF} for the definition of the (formal) fractional suspension.  Meanwhile, $\mathit{SWF}(\mathbb{RP}^3,\frs)=S^{\frac{1}{16}\mathbb{H}}$, and so we have a local map $\mathit{SWF}(Z)\to S^0$.  Using that $Z\cong -Z$, we have a local map $\mathit{SWF}(-Z)\to S^0$; furthermore, for general integer homology spheres $X$, we have $\mathit{SWF}(X)$ and $\mathit{SWF}(-X)$ are Spanier-Whitehead dual.  As a consequence, if there is a local map $\mathit{SWF}(-Z)\to S^0$ then there is a local map $S^0\to \mathit{SWF}(Z)$.  Thus $S^0\leq \mathit{SWF}(Z)\leq S^0$ in local equivalence, and so $\mathit{SWF}(Z)$ is locally trivial as a $\mathrm{Pin}(2)$-spectrum.
\end{proof}

\bibliographystyle{custom}
\def\MR#1{}
\bibliography{biblio}

\end{document}